\documentclass{amsart}
\usepackage{amsfonts, amsmath, amssymb, amscd, latexsym, graphicx}

\newtheorem{theorem}{Theorem}
\newtheorem{proposition}[theorem]{Proposition}
\newtheorem{lemma}[theorem]{Lemma}
\newtheorem{corollary}[theorem]{Corollary}
\newtheorem{claim}[theorem]{Claim}

\theoremstyle{definition}
\newtheorem{example}[theorem]{Example}
\newtheorem{definition}[theorem]{Definition}
\newtheorem{remark}[theorem]{Remark}

\newcommand\N{\mathbb N}

\newcommand\<{\langle}
\renewcommand\>{\rangle}
\newcommand\ZZ{\mathbb{Z}}

\newcommand\RR{\mathbb{R}}
\newcommand\bt{\begin{tabular}}
\newcommand\et{\end{tabular}}
\def\qed{\hfill\rlap{$\sqcup$}$\sqcap$\par}

\newcommand{\EGS}{EGS}
\newcommand{\Ereg}{MR2135549}
\newcommand{\Cannon}{MR758901}
\newcommand{\GrigTat}{MR1471889}
\newcommand{\NS}{MR1329042}
\newcommand{\Pascal}{MR1475566}
\newcommand{\Epstein}{MR1161694}
\newcommand{\Gromov}{MR623534}

\newcommand{\GromovHyp}{MR919829}

\begin{document}

\title{On groups whose geodesic growth is polynomial}

\author[M.R. Bridson]{Martin R. Bridson}
\address{Mathematical Institute, Oxford University, 24-29 St Giles', Oxford OX1 3LB, UK} \email{bridson@maths.ox.ac.uk}

\author[J. Burillo]{Jos\'e Burillo}
\address{
Departament de Matem\`atica Aplicada IV, EETAC, Universitat Polit\`ecnica de Catalunya,
08860 Castelldefels, Barcelona, Spain} \email{burillo@mat.upc.es}

\author[M. Elder]{Murray Elder}
\address{School of Mathematical and Physical Sciences, The University of Newcastle,
Callaghan, NSW 2308 Australia}
\email{murray.elder@newcastle.edu.au}

\author[Z.  \v Suni\'c]{Zoran \v Suni\'c}
\address{Department of Mathematics, Texas A\&M University, MS-3368, College Station, TX 77843-3368, USA}
\email{sunic@math.tamu.edu}

\subjclass[2000]{20F65}
\keywords{Geodesic growth; virtually nilpotent group;  virtually cyclic abelianization.}
\date{March 2011}

\begin{abstract}
This note records some
observations concerning geodesic growth functions.  If a nilpotent group is not virtually cyclic then it has exponential geodesic
growth with respect to all finite generating sets. On the other hand, if a finitely generated group $G$ has an element whose normal closure is
abelian and of finite index, then $G$ has a finite generating set with respect to which the geodesic growth is polynomial (this includes all virtually cyclic groups).
\end{abstract}

\maketitle

\section{Introduction}

Growth for finitely generated groups is a concept that has been studied extensively in the last fifty years, providing landmarks for modern group theory,
most notably Gromov's theorem characterizing groups that have  polynomial volume growth as virtually nilpotent groups~\cite{\Gromov}.
Volume growth functions count the number of elements in the ball of radius $n$ about the identity in the Cayley graph of a finitely generated group, while
{\em geodesic growth functions}
count the number of geodesics of length at most $n$ that begin at the identity
(cf.~\cite{\GrigTat}, \cite{\NS}, \cite{\Cannon}). The geodesic growth function of a group with respect to any finite generating set is bounded above by an exponential function. The purpose of the present note is to record some elementary observations concerning groups that have subexponential geodesic growth functions.

Previous work on geodesic growth functions has focussed mainly on the issue of
rationality of the associated formal power series. Gromov (\cite{\GromovHyp}, p.137) and Epstein {\em et al.} (\cite{\Epstein} p.80)  established
rationality for hyperbolic groups with respect to arbitrary finite generating sets.
There are similar rationality results for families of non-hyperbolic groups, but
examples of Cannon, described in~\cite{\NS},
show that rationality depends heavily on the choice of generators in the non-hyperbolic case.
Our results concerning polynomial geodesic growth\footnote{i.e. the geodesic growth function is bounded above by a polynomial} exhibit a
similar dependency.
Shapiro \cite{\Pascal}  considered the function $p_X:G\rightarrow \mathbb N$ which counts the number of geodesics for each group element, with respect to some finite generating set $X$. He gives an explicit formula for $p_X$ for abelian groups, from which a formula for the geodesic growth function can easily be obtained. We make use of his study of these functions in Section \ref{sec:zz} below.

In Section \ref{sec:defn} we  define geodesic growth carefully and present some basic properties and examples.
 In Section \ref{sec:zz} we prove that any nilpotent group $G$ that is not virtually cyclic has exponential geodesic growth with respect to all finite generating sets.  
 In the following sections we restrict our attention to
groups that have polynomial geodesic growth with respect to some generating set; such groups are virtually nilpotent
with virtually cyclic abelianization. In Section \ref{sec:vcyclic} we present two groups of the form $\ZZ^2\rtimes C_2$; both have virtually cyclic abelianization but one has exponential geodesic growth with respect to every generating set while the other has polynomial growth with respect to some finite generating sets.
In
Section~\ref{sec:main} we establish the following sufficient condition for polynomial geodesic growth.

\begin{theorem}[Main Theorem] \label{t:main}
Let $G$ be a finitely generated group.  If there exists an element $x\in  G$ whose normal closure is abelian and of finite index, then there exists a finite generating set for $G$ with respect to which the geodesic growth of $G$ is polynomial.
\end{theorem}

In Section \ref{s:six} we present a sharpening of this result in the virtually cyclic case.

\begin{theorem}\label{p:pe-for-vc}
Let $G$ be a virtually cyclic group generated by a finite symmetric set $X$.
The geodesic growth function $\Gamma_{G,X}$ is either bounded above and below by an exponential function,  or else is bounded above and below by polynomials of the same degree.
\end{theorem}

\section{Definition and elementary properties}\label{sec:defn}

Throughout this paper we consider groups $G$ equipped
with a finite generating set $X=\{x_1,\ldots,x_m\}$
that is symmetric (i.e. if $x\in X$ then
$x^{-1}\in X$).
We often think of $X$ as a subset of $G$ but there are times when
we must be more formal and regard the choice of
generators as an epimorphism $F(X)\to G$ from the free monoid generated by $X$; in particular this is necessary when we want to allow repetitions in our generating sets, i.e. allow that certain elements of $X$ have the same image under $F(X)\to G$.

The central concept of study in this paper is  \emph{geodesic growth}.

\begin{definition} The \emph{geodesic growth} of a group $G$ with respect to the symmetric generating set $X$ is the function $\Gamma_{G,X}(n)$ counting, for each $n$, the number of geodesics of length at most $n$ starting at 1 in the Cayley graph of the group $G$ with respect to $X$.
\end{definition}

There is an obvious upper bound on $\Gamma_{G,X}(n)$, namely $\Gamma_{G,X}(n)\le |X|^n$.
We say that  $\Gamma$ has {\em exponential} geodesic growth with respect to $X$ if there exists $b>1$ such that
$\Gamma_{G,X}(n)\ge b^n$ for all $n\in\N$. We say that $\Gamma$ has {\em polynomial} geodesic growth with respect to $X$ if there exist $c,d\in\N$ such that
$\Gamma_{G,X}(n)\le cn^d$ for all $n\in\N$. (At present, it is unclear if this is
equivalent to demanding upper and lower bounds of the same polynomial degree, cf.~\cite{MR623534}.)

The more usual (volume) growth, $\gamma_{G,X}(n)$,
counts the number of vertices in the Cayley graph of $G$ that are at distance at most $n$
from the identity. Since each vertex is connected to the identity by a geodesic, this function provides
a lower bound for
 $\Gamma_{G,X}(n)$.

\begin{lemma}\label{p:lower-bound}
Let $G$ be a group with finite generating set $X$. The geodesic growth function $\Gamma_{G,X}(n)$
is bounded below by the word growth function $\gamma_{G,X}(n)$: for all $n\in\N$,
\[ \Gamma_{G,X}(n) \geq \gamma_{G,X}(n). \]
\end{lemma}

The following examples show that geodesic growth is heavily dependent on the choice of generating set.

\begin{example}\label{eg1}Consider the group $G=\ZZ\times C_2$. (Throughout this paper, $C_k$ denotes the cyclic group of order $k$.) We present $G$ as $\<t,a| a^2=1, at=ta\>$.  With respect to the symmetric
generating set  $\{t^{\pm 1}, a\}$, the geodesics of length $n$ are $t^n$, $t^{-n}$, $t^iat^{n-i-1}$, and $t^{-i}at^{i+1-n}$, for $i=0,\dots,n-1$. Thus the number of geodesics of length $n$ is, for $n \geq 2$, equal to $2n+2$, so the number of geodesics of length at most $n$ is $O(n^2)$.

Now consider  the presentation $\<t,c|c^2=t^2,ct=tc\>$, obtained from the previous one by the substitution $c=at$. With respect to the generating set
$\{t,t^{-1},c, c^{-1}\}$ each word
$x_1x_2\ldots x_n$ with $x_i\in \{tt,cc\}$ is a geodesic for the element $t^{2n}$. The number of such strings
for each $n$ is $2^n$, so the geodesic growth is exponential.  \end{example}
\begin{figure}[!ht]
 \centering
\bt{c}
 \includegraphics[scale=0.5]{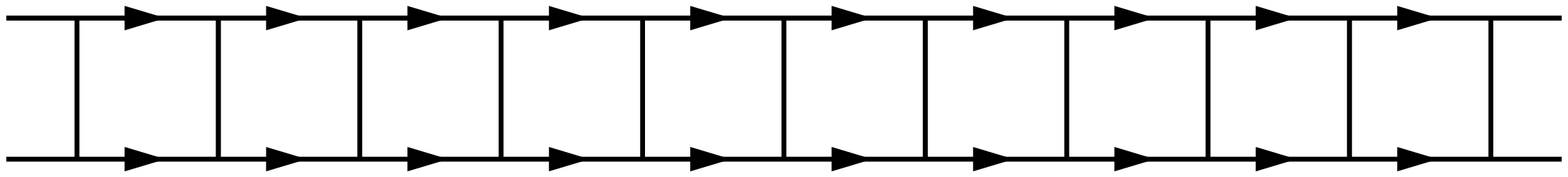}\\
 $\;$\\
 \includegraphics[scale=0.5]{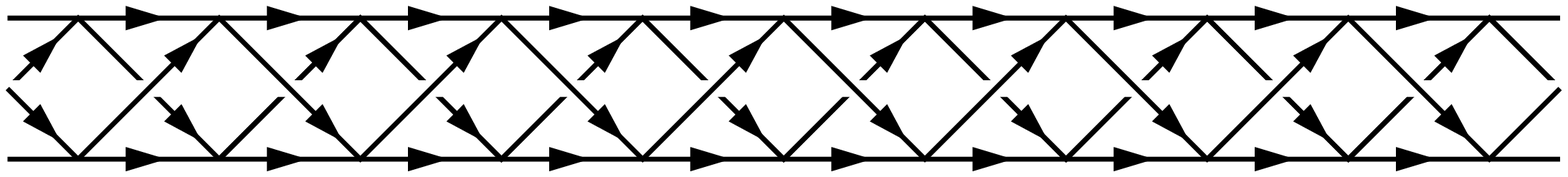}\et
  \caption{Cayley graphs for Example~\ref{eg1}.}
 \label{}
\end{figure}
Note that both generating sets are minimal: one cannot generate the group with fewer than two generators.

\begin{example} Consider $\ZZ$ with the presentation $\<t \mid\  \>$ and generating set $\{t^{\pm 1}\}$. For $n=0$ there is one geodesic of length $n$, and for $n>0$ there are exactly two, so the geodesic growth function is linear. Now consider the presentation $\<t,s \mid t=s\>$. With respect to the generating set $\{s^{\pm 1}, t^{\pm 1}\}$, there are $2^n$ geodesics joining 1 to the group element $t^n$, namely those labelled by positive words of length $n$ in the symbols $s$ and $t$.
\end{example}

A similar doubling trick shows that, with respect to some finite generating set, the geodesic growth of every finitely generated infinite group is exponential. Thus if one wants to make non-trivial statements about
groups with subexponential geodesic growth, then one has to be content with imposing this constraint
with respect to {\em some} finite generating set (not an arbitrary one).
We would like to understand which groups
have subexponential (polynomial or intermediate) geodesic growth in this sense.

The bound recorded in Lemma~\ref{p:lower-bound} tells us that if a group has polynomial geodesic growth with respect to some generating set then its word growth must be polynomial. Hence, by Gromov's celebrated theorem~\cite{\Gromov}, we obtain the following result which explains why we focus on virtually
nilpotent groups in the rest of the text.

\begin{corollary}
If a group $G$  has polynomial geodesic growth with respect to some finite generating set, then $G$  is virtually nilpotent.
\end{corollary}

In addition to the obvious lower bound provided by the word growth function given in Lemma~\ref{p:lower-bound} there is another lower bound that comes from the good behavior of geodesics under lifts along homomorphisms.

\begin{lemma}\label{p:quotient-bound}
Let $G$ be a group with a finite symmetric generating set $X$.
Let $\phi: G \to G'$ be an epimorphism of groups and take $X'=\phi(X)$ as a generating set for $G'$. The geodesic growth functions of $G$ and $G'$ satisfy the following inequality:  for all $n \geq 0$,
\[
 \Gamma_{G,X}(n) \geq \Gamma_{G',X'}(n).
\]
\end{lemma}

\begin{proof} If $w=\phi(x_1)\dots\phi(x_n)$ is a geodesic word in $G'$, then $x_1\dots x_n$
 must be a geodesic in $G$: if there were a shorter word $y_1\dots y_l$ evaluating to
 the same element of $G$, then we would have $w=\phi(y_1)\dots \phi(y_l)$ in $G'$, contradicting the
 assumption that $w$ is geodesic. Thus a choice of set-theoretic section $X'\to X$ of
 $x\mapsto \phi(x)$ defines an injection from the set of geodesic words of length $n$ in $(G',X')$ to
 the set of geodesic words of length $n$ in $(G,X)$.
\end{proof}

\begin{remark}
In this article we will not investigate the possible existence of
finitely generated groups that, with respect to certain generating
sets, have geodesic growth that is sub-exponential but
super-polynomial, i.e. ``intermediate".  Groups with intermediate word
growth have attracted considerable attention since their discovery by
Grigorchuk \cite{MR764305} in the 1980s. Might some such groups have intermediate
geodesic growth with respect to some generating sets?
It is natural to look first at Grigorchuk's original group.
 With respect to the standard generating set of four
involutions, the third and fourth authors, working with Gutierrez
\cite{\EGS}, have proved that its geodesic growth rate lies between
$O((\sqrt 2)^n)$ and $O((\sqrt{1+\sqrt 3})^n)$; in particular it is not
intermediate.

{\em A priori}, it is also possible that there is an example of a
virtually nilpotent group that has intermediate geodesic growth  with respect to some finite generating set.
\end{remark}


\section{Geodesic growth for groups that map onto $\ZZ^2$}\label{sec:zz}

In this section we show that if a group maps onto $\ZZ^2$,
then its geodesic growth with respect to any finite generating set is
exponential; any finitely generated group that is not virtually abelian
satisfies this condition.
 Our proof closely follows work of Michael Shapiro (Section 2 in \cite{\Pascal}). We thank Mark Sapir for pointing this out to us.
We  show that if a finitely generated nilpotent group is not virtually cyclic, then it satisfies
this condition.

Take a basis $\{a,b\}$ of  $\ZZ^2$ and work with the symmetric generating set
$X=\{a^{\pm 1}, b^{\pm 1}\}$. The  $\binom{2n}{n}$ distinct permutations of the word $a^nb^n$ provide us with more than $2^n$ geodesics of length $2n$, so the geodesic growth of $\ZZ^2$ with respect to $X$ is exponential. The following proposition extends this observation to arbitrary finite generating sets of $\ZZ^2$. 
\begin{proposition}[Shapiro \cite{\Pascal}]\label{free_abelian_case}
The group $\ZZ^2$ has exponential geodesic growth with respect to every finite generating set.
\end{proposition}
For completeness, we present a self-contained proof.
\begin{proof}
Let $X=\{x_1,\ldots,x_m\}$ be a symmetric generating set
for $\ZZ^2$. We embed $\ZZ^2$ in $\RR^2$ as the set of points with integer coordinates and
write $x_i=(a_i,b_i)\in \RR^2$.

Consider the non-degenerate, centrally-symmetric
convex polygon $P\subset\RR^2$ that is the convex hull of $\{ (a_i, b_i) \mid i=1,\ldots,m \}$.
We write $\lambda P$ for the image of $P$ under the homothety $v\mapsto \lambda v$ of $\RR^2$.

 Let $x_i, x_j\in X$ be such that the segment joining $(a_i,b_i)$ to
$(a_j,b_j)$
lies entirely on the boundary of $P$.
We claim that  $x_i^nx_j^n$ is at distance $2n$
from the identity in the  Cayley graph of  $\ZZ^2$ with respect to $X$.
To see this, note that the vectors that represent the elements of $\ZZ^2$ of length at most $2n-1$ with respect to $X$ are contained in the polygon $(2n-1)P$, and the vector $nx_i + nx_j$ is on the middle of the edge $[2nx_i,2nx_j]$ of the polygon $2nP$, and this edge is entirely outside of $(2n-1)P$.

Thus the word $x_i^nx_j^n$ is geodesic of length $2n$, as is each of
 the $\binom{2n}{n}$ distinct words obtained by
permuting its letters.
\end{proof}

By applying Lemma~\ref{p:quotient-bound} we deduce:

\begin{corollary} \label{p:onto-z^2}
A  group that maps homomorphically onto $\ZZ^2$ has exponential geodesic growth with respect to any finite generating set.
\end{corollary}

\begin{proposition}\label{t:nilpotent}
Let $G$ be a finitely generated nilpotent group. If $G$ is not virtually cyclic, then
it has exponential growth with respect to every finite generating set.
\end{proposition}

In the light of Corollary \ref{p:onto-z^2}, this is an immediate consequence of the
following lemma.

\begin{lemma}\label{l:nilp}
If a  finitely generated nilpotent group is not virtually cyclic, then it maps homomorphically onto $\ZZ^2$.
\end{lemma}

\begin{proof} If $G$ is virtually abelian, then $G$ modulo its (finite) torsion subgroup is free abelian.
This observation covers the base case of an induction on the nilpotency class of $G$.
In the inductive step we can assume that $G$ modulo its centre $Z$ satisfies the lemma.
If $G/Z$ maps onto $\ZZ^2$ then so does $G$. If $G/Z$ is virtually cyclic then $G$,
being a central extension of a virtually cyclic group, is virtually abelian and the observation
in the first sentence of the proof applies once more.
\end{proof}


\section{Geodesic growth in some groups with virtually cyclic abelianization}\label{sec:vcyclic}

According
to Corollary~\ref{p:onto-z^2}, the only groups that might have finite generating sets with respect to which the geodesic growth is polynomial are those with virtually cyclic abelianization. Here we consider
two such groups: both are of the form $\ZZ^2\rtimes C_2$ but their geodesic growth functions behave very
differently.

\begin{remark} Let $G$ be a group that has a finite index subgroup $H$ mapping onto a
free abelian group $A$ of rank at least 2. By  inducing  $H\to A$ (in the
sense of representation theory) we obtain a
map from $G$ onto a virtually abelian group of rank at least 2. Thus Lemma \ref{l:nilp}
 implies that if a finitely generated {\em virtually} nilpotent group $G$ is not virtually cyclic, then it  maps onto a  group that is {\em virtually} free abelian group of rank at least 2.  Naively, one might hope that  this extension of Lemma \ref{l:nilp} would  lead to an exponential lower bound on the geodesic growth  functions of $G$, thus completing the characterisation of groups with polynomial geodesic growth. But the following example shows that this is not the case.
\end{remark}

\begin{example}\label{ex:ZbyC}
Let $\phi_1$ be the automorphism of $\ZZ^2$ that interchanges the basis elements $a$ and $b$,
and consider
\[
G_1 = \ZZ^2 \rtimes_{\phi_1} C_2 = \langle \ a,b,t \mid [a,b]=1, \ t^2=1, \ a^t=b \ \rangle,
\] which has abelianization $\ZZ \times C_2 $.

With respect to the generating set $\{a,a^{-1},t\}$ the
 language of geodesics has been computed explicitly~\cite{\Ereg} to be the set of words \[\left\{a^x, \ a^xta^y, \ a^{x_1}ta^{y_1}ta^{x_2}  \mid  x,y,x_1,y_1,x_2\in \ZZ,  \ x_1\cdot x_2 \geq 0, \ |y_1|>0\right\}.\]
Thus the geodesics of length $n\geq 5$ are
\[\begin{array}{lll}
a^{\pm n}, \\
a^{\pm (n-1)}t,\\
 ta^{\pm (n-1)}, \\
   a^{\pm i}ta^{\pm (n-i-1)} & \mathrm{for} &    i=1,\ldots, n-2, \\
ta^{\pm (n-2)} t, \\
 a^{\pm i}ta^{\pm (n-2-i)}t & \mathrm{for} &  i=1,\ldots, n-3, \\
ta^{\pm i}ta^{\pm (n-2-i)} & \mathrm{for} & i=1,\ldots, n-3,\\
 a^{\epsilon i}ta^{\pm j}t a^{\epsilon (n-i-j-2)} & \mathrm{for} & i=1,\ldots, n-4, \\
&& j=1,\ldots, n-i-3 \ \mathrm{and} \ \epsilon=\pm 1\end{array}\]
of which there are \[2+2+2+4(n-2)+2+4(n-3)+4(n-3)+4\sum_{i=1}^{n-4}i=2n^2-2n.\]
\end{example}

Similar examples may be constructed that are virtually free-abelian of arbitrary rank.

\begin{example}\label{ex:Zoran}
Let $\phi_2$ be the automorphism that sends each of the basis elements $a$ and $b$
to their inverse,
and consider
\[
 G_2 = \ZZ^2 \rtimes_{\phi_2} C_2 = \langle \ a,b,t \mid
 [a,b]=1, \ t^2=1, \ a^t=a^{-1}, \ b^t=b^{-1} \ \rangle,
\]
which has  abelianization $C_2\times C_2\times C_2$.

Let $H=\langle a,b\rangle$.  Note that if $h\in H$ and $g\in G_2\smallsetminus H$ then $h^g = h^{-1}$
and $g^2=1$.

Suppose $X$ is a symmetric generating set for $G_2$. We split $X$ as a disjoint union $X=Y \cup Z$,
where  $Y=\{x \in X \mid x \not \in H\}$ and $Z=\{x \in X \mid x \in H\}$.

Let $S=\{z^2 \mid z \in Z\} \cup \{(yy') \mid y,y'\in Y\}$. Note that $S$ is a symmetric set (since $Z$ is symmetric and the $Y$-letters have order 2 in $G_2$).

By extending the map $a \mapsto (1,0), \ b \mapsto (0,1)$ to a homomorphism, we can, as in the proof of Proposition \ref{free_abelian_case}, embed $H=\ZZ^2$ in $\RR^2$
as the set of points with integer coordinates. For $h \in H$, denote by $\overline{h}$ the vector in $\RR^2$ that is the image of $h$ under this embedding. Let $P$ be the centrally symmetric, convex polygon $P\subset\RR^2$ that is the convex hull of $\overline{S} = \{ \overline{s} \mid s \in S \}$. As before, $\lambda P$ denotes the image of $P$ under the homothety $v\mapsto \lambda v$ of $\RR^2$.

We claim that if $h\in H$ is at distance $n$ from the identity in the word
metric associated to $X$, then $\overline{h}$ is in $(n/2)P$.
To see this, consider a geodesic factorization $h=x_1x_2...x_n$.
We rewrite this factorization by pushing the $Y$-letters to the left, inverting
the $Z$-letters as they are pushed past using the relations $zy=yz^{-1}$
where necessary. The number of $Y$-letters in any representation of $h$ is even,
so at the end of this process we have a geodesic factorisation
\begin{equation}\label{e:h}
 h =  (y_1 y_2) \dots (y_{2m-1} y_{2m}) z_{2m+1} z_{2m+2} \dots z_n.
\end{equation}
Hence
\[
 \overline{h} = \overline{y_1 y_2} + \dots + \overline{y_{2m-1} y_{2m}} + \frac{1}{2}\overline{z_{2m+1}^2} + \dots + \frac{1}{2}\overline{z_n^2}.
\]
Thus $\overline{h}$ is a positive linear combination of vectors in $S \subseteq P$ with
the sum of coefficients equal to $n/2$; in particular, $\overline{h} \in (n/2) P$.

Let $s_1, s_2\in S$ be such that the segment joining the vertex $\overline{s}_1$ to $\overline{s}_2$ in $P$ lies entirely on the boundary of $P$, and consider $g=s_1^ns_2^n$.

We claim that  $s_1^ns_2^n$ is a geodesic for $g$ of length $4n$ with
respect to $X$. Indeed we have just seen the elements of $H$ that are a distance at most $4n-1$ from
the identity determine vectors that
 are contained in the polygon $(2n-\frac12)P$, whereas
 $\overline{s_1^ns_2^n} = n\overline{s}_1 + n\overline{s}_2$ lies on the boundary of $(2n)P$.

Since $s_1$ and $s_2$ commute, there are at least $\binom{2n}{n}$ geodesics for $g$ with respect to $X$,
showing that the geodesic growth of $G$ with respect to $X$ is exponential.
\end{example}


\section{Proof of the Main Theorem}\label{sec:main}

Our proof of the Main Theorem begins with  two lemmas, as follows.

Fix $\theta\ge 1$. Given a group $G$
with a symmetric, finite generating set $S$, we say a word $w=a_1\ldots a_l$ in the letters $S$ is  {\em{$\theta$-efficient}} if, in the associated word metric,
 $l\le\theta d(1,w) $. 

\begin{lemma}[Highways beat byways]\label{lemmama1} Fix $\theta \ge 1$.
Let $A$ be a finitely generated free
abelian group, let
$S$ be a symmetric, finite generating set for $A$ and let $T$ be a finite subset of $A$
such that $\langle T\rangle$ has finite index in $A$. Let $N\in\N$ and consider the set $S\cup T_N$, where
\[
T_N=\{t^{\pm N}\,|\,t\in T\}.
\]
If $N$ is sufficiently large  then only finitely many words in the letters
$S$ are $\theta$-efficient for the word metric associated to $S\cup T_N$.
\end{lemma}

\begin{proof} We identify $A$ with the set of points in $\RR^r$ with integer coordinates
and equip $\RR^r$ with the standard Euclidean norm $\|\cdot\|$.
The proof involves this norm, the word metrics $d_S$ and $d_{S\cup T_N}$
on $A$, and $d_T$ and $d_{T_N}$ on $\<T\>$ and $\<T_N\>$ respectively. We write $0$ for the identity element of $A\subset\RR^r$.

We fix
 $\alpha,\beta>0$ such that, for all $x\in A$,
\[
\beta ||x|| \leq d_S(0,x) \leq \alpha||x||,
\]
and $\varepsilon>0$ such that for all $y\in \RR^r$, $\|y- [y]_T\|\le \varepsilon$, where
$[y]_T$ is the lex-least among the points in $\langle T\rangle\subset\ZZ^r\subset\RR^r$
that are nearest to $y$.
 Finally, fix $\lambda>0$ such that for all $y\in \RR^r$ 
\[
d_T(0,[y]_T)\leq\lambda ||y||.
\]
Suppose that $N$ is large enough so that $\beta/\theta-\lambda/N>0$.

Observe that, for all $z\in \RR^r$, we have:
\begin{enumerate}
\item $[\frac1N z]_T=\frac1N[z]_{T_N}$,
\item $\|z-[z]_{T_N}\|=N\|\frac1N z-[\frac1N z]_T\|$,
\item $d_T(0,[\frac1N z]_T)=d_{T_N}(0,[z]_{T_N})$,
\end{enumerate}
whence $\|z-[z]_{T_N}\|\le N\varepsilon$.

For all $z\in A=\ZZ^r$,
\begin{align*}
d_{S\cup T_N}(0,z)
 &\leq d_{S \cup T_N}(0,[z]_{T_N})+d_{S \cup T_N}(z,[z]_{T_N})\\
 &\leq d_{T_N}(0,[z]_{T_N})+d_S(z,[z]_{T_N})\\
 &= d_T(0,[\tfrac1N z]_T)+d_S(0,z-[z]_{T_N})\\
 &\leq \lambda\|\frac{1}{N}z\| + \alpha\|z-[z]_{T_N}\|\\
 &\leq \frac{\lambda}{N}||z||+\alpha N \varepsilon.
\end{align*}
So if $z\in A$  has a $\theta$-efficient
 representative   with respect to $S \cup T_N$
that contains no letters from $T_N$, then
\[ \beta||z|| \leq d_S(0,z) \le\theta d_{S\cup T_N}(0,z) \leq
\theta\left(\frac{\lambda}{N}||z||+\alpha N \varepsilon\right),
\]
yielding
\[
||z||(\beta/\theta-\lambda/N)\leq \alpha N \varepsilon
\]
or, equivalently (since $\beta/\theta - \lambda/N>0$),
\[
||z||\leq\frac{\alpha N \varepsilon}{\beta/\theta-\lambda/N}.
\]
There are only finitely many  elements of $A=\ZZ^r$ that satisfy this bound.
For each such element, the number of 
$\theta$-efficient representatives is finite since their length is at most $$\theta d_{S\cup T_N} (0,z)\leq \theta\left(\frac{\lambda}{N}||z||+\alpha N \varepsilon\right)\leq \theta\left(\frac{\lambda}{N}\left(\frac{\alpha N \varepsilon}{\beta/\theta-\lambda/N}\right)+\alpha N \varepsilon\right).$$ The result follows.
\end{proof}

\begin{lemma}[Dominating generator]\label{lemmama2}
For a group $G$, let $X=\{x_0,x_0^{-1}\}\cup Y$ 
be a finite symmetric generating set. If there exists a constant $k$ such that, for any geodesic in $G$ with respect to $X$, the total number of appearances of the generators $y\in Y$
 is at most $k$, then the geodesic growth of $G$ with respect to $X$ is polynomial.
\end{lemma}

\begin{proof} Let $m=|Y|$.
Each geodesic word of length $n$ over $X$ has the form
\[
 x_0^{\pm n_0} y_1 x_0^{\pm n_1} y_2 \dots y_\ell x_0^{\pm n_\ell},
\]
with  $\ell \leq k$, where $y_i\in Y$ and $n_0,\dots,n_\ell$ are non-negative integers such that
\[ n_0+\dots+n_\ell=n-\ell. \]
To construct any word of this form, one can start with the string $x_0^n$ and choose the $\ell$
sites at which to replace $x_0$ by some $y\in Y$; one then has to make a choice of sign for
the remaining strings of $x_0$. For fixed $\ell$, the number of possibilities for the
set of $Y$-sites is $\binom{n}{\ell}$ (the number of non-negative integer partitions of $n-\ell$). And
there are at most  $2^{\ell+1}$ possible sign choices
(fewer if some of the integers $n_0,\dots,n_\ell$ are 0). Therefore, the number  of geodesics
of length $n$, for $n >> k$, is bounded above by
\[
 2^{k+1}m^k \binom{n}{k} + 2^km^{k-1} \binom{n}{k-1} + \dots + 2^2m \binom{n}{1} + 2 \leq f(n):= (k+1)2^{k+1}m^k\binom{n}{k}.
\]
If $n$ is sufficiently large, then $f(n')<f(n)$ for all $n'<n$
and hence the number of geodesics of length at most $n$ is bounded above by
\[
 (n+1) f(n)  = 2^{k+1}m^k(k+1) (n+1)\binom{n}{k},
\]
which is a polynomial of degree $k+1$.
\end{proof}

\subsection*{Proof of the Main Theorem}
Let $G$ be a finitely generated group that contains an element $x$ whose normal closure
$A=\<\!\<x\>\!\>$ is abelian and of finite index. $A$ is finitely generated,
since it has finite-index in $G$. Replacing $x$ by a proper power if necessary,
we may assume that $A$ is free abelian of finite rank.

Let $Q=G/A$ and let
$\pi:G\longrightarrow Q$
be the natural projection. We
fix a set of coset representatives $R=\{q_1,q_2,\ldots,q_l\}$ for $A$ in $G$ and define
\[
D=\{pqr^{-1}\,|\,p,q,r\in R\cup R^{-1},\text{and }\pi(r)=\pi(p)\pi(q)\}.
\]
Let $S\subset A$ be a symmetric generating set that includes $D\subset A$ and is invariant under the conjugation action of $G$. 
The generators in $S$ are thought of as \emph{short} generators.

Define $X=S\cup\{x^N,x^{-N}\}\cup R \cup R^{-1}$ as a finite generating set for $G$, for some $N\in \mathbb N$. 
We will show that 
if $N$ is sufficiently large, then $G$ has polynomial growth with respect to $X$. 

First, we apply Lemma \ref{lemmama1} to $A$
 with $T:=\{q_ixq_i^{-1} \mid i=1,\dots,\ell\}$ in order to deduce the following:

\begin{claim}\label{obsv} If $N$ is sufficiently large, then
there is a constant $k=k(N)$ such that any word over $S$ that is geodesic
with respect to $X$  has length at most $k$.
\end{claim}

To justify this observation, we need to know that a word in the letters $S$
that is geodesic in $(G,d_X)$ is  $\theta$-efficient in $(A,d_{S\cup T_N})$, where $\theta$ does not
depend on $N$. But  for each
$t=q_ixq_i^{-1}$ we have  $d_X(1,t^N) \le d_X(1,x^N) + 2 d_X(1,q_i) =3$, so if $a\in A$
equals a word of length $L$ in the letters $S\cup T_N$, then
it equals a word of length at most $3L$ in the letters $X$. In particular, words in $S$
that are geodesic in $(G,d_X)$ are 3-efficient in  $(A,d_{S\cup T_N})$.

We are now ready for the main argument of the proof. To avoid confusion, we write $y=x^N$.
In the light of Lemma~\ref{lemmama2}, it suffices to prove that
the number of letters from $S \cup R \cup R^{-1}$ in any geodesic over $X$ is uniformly bounded;
we shall prove that it is bounded by $k + |Q|$.

Given a geodesic word $w$ over $X$ we may do the following.
\begin{enumerate}
\item Since $S$ is invariant under conjugation by $R\subset G$ and its elements commute with $y$, all occurences
in $w$ of letters $s\in S$  can be moved to the left without changing the length of $w$. (As
$s\in S$ is pushed past $q_i\in R$ it is replaced by $q_isq_i^{-1}\in S$.)
\item If a subword of the form $pq$ with $p,q\in R \cup R^{-1}$ appears in $w$
at any stage, then we can replace it by $dr$, where $d=pqr^{-1}\in D$. Then,
 since $d\in D\subset S$, we can move $d$ to the left.
\end{enumerate}
Proceeding in this manner, an arbitrary geodesic over $X$ can be transformed into one of the form
\[
u(S)q_{i_1}y^{n_1}q_{i_2}y^{n_2}\ldots q_{i_\lambda}y^{n_\lambda},
\]
with $q_{i_j}\in R^{\pm 1}$, where $u(S)$ is a word over $S$, and where $\lambda$ has been made as
small as possible. We claim that  $\lambda\le l = |Q|$.

To prove this claim, consider what would happen
if there were more than $l$ factors. There would be two prefixes
\[
p_a=q_{i_1}y^{n_1}q_{i_2}y^{n_2}\ldots q_{i_a},\qquad p_b=q_{i_1}y^{n_1}q_{i_2}y^{n_2}\ldots q_{i_b},
\]
for some $a < b\leq \lambda$ (including the possibility $a=0$, i.e, $p_a$ being the empty word), such that $\pi(p_a^{-1}p_b)=1$ in $Q$. Hence,  $p_a^{-1}p_b=y^{n_a}w\in A,$
where $w=q_{i_{a+1}}y^{n_{a+1}}\ldots q_{i_b}$. Since $y^{n_a}$ commutes with $w$, we could
 bring $q_{i_a}$ and $q_{i_{a+1}}$ next to each other and perform a move of type (2)  (including the move of the newly obtained $d$ from $D$ all the way to the left). But this would contradict the assumption that
$\lambda$ had been minimized.

 Claim \ref{obsv} implies that the length of $u(S)$ is at most $k$, so the number of letters different from $y^{\pm 1}$ in the modified geodesic is at most $k+l$.
Since moves (1) and (2) do not change the number of $y^{\pm 1}$ letters, the number of letters different from $y^{\pm 1}$ in the modified geodesic is the same as that in the original geodesic.
 This completes the proof.
\qed

\section{Virtually cyclic groups}\label{s:six}

In the context of the Main Theorem, we do not know if one can obtain upper and lower bounds
of the same polynomial degree. But in the case of virtually cyclic groups, one can obtain
such bounds.

\setcounter{theorem}{1}

\begin{theorem}
Let $G$ be a virtually cyclic group generated by a finite symmetric set $X$. The geodesic growth function $\Gamma_{G,X}$ is either bounded above and below by an exponential function,  or else is bounded above and below by polynomials of the same degree.
\end{theorem}

\begin{proof} The exponential upper bound is trivial, and
we observed in Section~\ref{sec:defn} that every infinite finitely generated group has some finite generating set for which the geodesic growth is exponential.

Theorem \ref{t:main} shows that virtually cyclic groups have at least one generating set for which the geodesic growth is polynomial. A virtually cyclic group is hyperbolic, and the language of geodesics is regular  for every finite generating set (see~\cite{MR1161694}). If the growth of a regular language is subexponential, then the language is simply starred and its growth is bounded above and below by polynomials of the same degree
(see~\cite{MR1906807} and \cite{MR1161694} p.20).
\end{proof}

\section*{Acknowledgments}

The first author is supported by a Senior Fellowship from the EPSRC.
The second author acknowledges
support from MEC grant MTM2008-01550 and is grateful for
the hospitality of the University of Queensland.
The third author acknowledges
support from
the University of Queensland NSRSF project number 2008001627 and
 ARC grants DP110101104 and FT110100178. The fourth author acknowledges support by NSF grants  DMS-0805932 and DMS-1105520.

\end{document}